\documentclass[letterpaper,11pt]{amsart}

\usepackage[margin=1.2in]{geometry}
\usepackage{amsmath,amsthm,amssymb}
\usepackage{xspace,xcolor}
\usepackage[breaklinks,colorlinks,citecolor=teal,linkcolor=teal,urlcolor=teal,pagebackref,hyperindex]{hyperref}
\usepackage[alphabetic]{amsrefs}
\usepackage[all]{xy}
\usepackage{color}

\theoremstyle{plain}
\newtheorem{thm}{Theorem}[section]

\newtheorem{prop}[thm]{Proposition}
\newtheorem{cor}[thm]{Corollary}

\theoremstyle{definition}

\newtheorem{eg}[thm]{Example}

\newtheorem{conjecture}[thm]{Conjecture}

\theoremstyle{remark}
\newtheorem{rmk}[thm]{Remark}

\def\Z{{\mathbf Z}}

\def\R{{\mathbf R}}
\def\C{{\mathbf C}}

\def\A{{\mathbf A}}

\def\cD{\mathcal{D}}

\def\cJ{\mathcal{J}}

\def\cO{\mathcal{O}}

\def\fra{\mathfrak{a}}

\def\.{\cdot}
\def\^{\widehat}

\def\({\left(}
\def\){\right)}

\renewcommand{\and}{ \ \ \text{ and } \ \ }

\newcommand{\factor}[2]{\left. \raise 2pt\hbox{$#1$} \right/\hskip -2pt\raise -2pt\hbox{$#2$}}

\DeclareMathOperator{\lct} {lct}
\DeclareMathOperator{\fpt}{fpt}

\makeatletter
\@namedef{subjclassname@2020}{\textup{2020} Mathematics Subject Classification}
\makeatother

\usepackage{soul}

\begin{document}

\author[M.~Musta\c{t}\u{a}]{Mircea Musta\c{t}\u{a}}

\address{Department of Mathematics, University of Michigan, 530 Church Street, Ann Arbor, MI 48109, USA}

\email{mmustata@umich.edu}

\thanks{The author was partially supported by NSF grants DMS-2001132 and DMS-1952399.}


\begin{abstract}
Given a smooth, irreducible complex algebraic variety $X$ and a nonzero regular function $f$ on $X$, we give an effective estimate for the difference between 
the jumping numbers of $f$ and the $F$-jumping numbers of a reduction $f_p$ of $f$ to characteristic $p\gg 0$, in terms of the roots of the
Bernstein-Sato polynomial $b_f$ of $f$. In particular, we get uniform estimates only depending on the dimension of $X$.
As an application, we show that if $b_f$ has no roots of the form $-\lct(f)-n$, with $n$ a positive 
integer, then the $F$-pure threshold of $f_p$ is equal to the log canonical threshold of $f$ for $p\gg 0$ with $(p-1)\lct(f)\in\Z$. 
\end{abstract}

\title[An estimate for $F$-jumping numbers]{An estimate for $F$-jumping numbers via the roots of the Bernstein-Sato polynomial}

\maketitle

\section{Introduction}

Let $X$ be a smooth, irreducible complex algebraic variety and let $f\in\cO_X(X)$ be a nonzero regular function defining the hypersurface $H\subset X$.
One associates to $f$ a sequence of coherent ideals $\cJ(f^{\lambda})\subseteq\cO_X$, the \emph{multiplier ideals} of $f$, depending on the parameter $\lambda\in\R_{\geq 0}$.
They can be defined either in terms of a log resolution of the pair $(X,H)$, or in terms of integrability conditions. These ideals give interesting invariants of the singularities of the hypersurface $H$ and they play an important role in vanishing theorems. They satisfy $\cJ(f^{\lambda})\subseteq \cJ(f^{\mu})$ if $\mu\leq\lambda$ and there is an
increasing sequence of positive rational numbers $(\lambda_m)_{m\geq 1}$, with $\lim_{m\to\infty}\lambda_m=\infty$, such that $\cJ(f^{\lambda})$ is constant for
$\lambda\in [\lambda_{m-1},\lambda_m)$ for all $m\geq 1$ (with the convention that $\lambda_0=0$). These are the \emph{jumping numbers} of $f$ and the smallest
jumping number $\lambda_1$ is the \emph{log canonical threshold} $\lct(f)$. 
For an introduction to multiplier ideals and jumping numbers, see \cite[Chapter~9]{Lazarsfeld}. 

It has long been understood that many classes and invariants of singularities that appear in birational geometry in characteristic 0 have analogues in positive characteristic,
defined via the Frobenius morphism. Suppose now that $Y$ is a regular scheme of characteristic $p>0$, which we assume to be $F$-finite (this means that the
Frobenius morphism $F\colon Y\to Y$ is a finite morphism). If $g\in\cO_Y(Y)$ is everywhere nonzero, then Hara and Yoshida \cite{HY} defined a sequence of ideals
$\tau(g^{\lambda})\subseteq \cO_Y$, the (generalized) \emph{test ideals} of $g$, depending on the parameter $\lambda\in\R_{\geq 0}$. These ideals turn out to satisfy many of the 
formal properties that multiplier ideals satisfy in characteristic $0$. In particular, $\tau(g^{\lambda})\subseteq \tau(g^{\mu})$ if $\mu\leq\lambda$ and it was shown in 
\cite{BMS2} that there is an
increasing sequence of positive rational numbers $(\alpha_m)_{m\geq 1}$, with $\lim_{m\to\infty}\alpha_m=\infty$, such that $\tau(g^{\lambda})$ is constant for
$\lambda\in [\alpha_{m-1},\alpha_m)$ for all $m\geq 1$ (with the convention that $\alpha_0=0$). These are the \emph{$F$-jumping numbers} of $g$ and the smallest such
invariant $\alpha_1$ is the \emph{$F$-pure threshold} $\fpt(g)$, introduced by Takagi and Watanabe in \cite{TW}. 

The most interesting results and open problems in this area are related to the comparison between multiplier ideals and test ideals via reduction mod $p$. Suppose that $X$
is a smooth, irreducible complex algebraic variety and $f\in\cO_X(X)$ is nonzero. After choosing a model $(X_A,f_A)$ of $(X,f)$ over a finitely generated $\Z$-subalgebra $A\subseteq\C$, 
for every closed point 
$t\in {\rm Spec}(A)$ we obtain a pair $(X_t,f_t)$ in positive characteristic (note that the field $k(t)$ is a finite field). As it is typical in this setting,
we always allow replacing $A$ by a localization $A_a$, in order to preserve certain properties from characteristic $0$ (for example, in this way 
we may assume that $X_A$ is smooth over ${\rm Spec}(A)$ and that $f_A$ defines a relative Cartier divisor over ${\rm Spec}(A)$).
The following is a key result from \cite{HY}:

\begin{thm}\label{thm0}
With the above notation, after possibly replacing $A$ by a localization $A_a$, the following hold:
\begin{enumerate}
\item[i)] For every closed point $t\in {\rm Spec}(A)$, we have
$\tau(f_t^{\lambda})\subseteq \cJ(f^{\lambda})_t$ for all $\lambda\in\R_{\geq 0}$.
\item[ii)] For every $\lambda\in\R_{\geq 0}$, if $t\in {\rm Spec}(A)$ is a closed point such that ${\rm char}\,k(t)\gg 0$, then
$\tau(f_t^{\lambda})=\cJ(f^{\lambda})_t$.
\end{enumerate}
\end{thm}

The assertion in ii) is the deepest one.
We note that the condition on ${\rm char}\,k(t)$ here can't be made independent of $\lambda$, in general. However, we have the following
open problem concerning the relation between multiplier ideals and test ideals:

\begin{conjecture}\label{conj1}
With the above notation, there is a dense subset $T$ of closed points in ${\rm Spec}(A)$ such that
$$\tau(f_t^{\lambda})=\cJ(f^{\lambda})_t\quad\text{for all}\quad t\in T,\lambda\in\R_{\geq 0}.$$
\end{conjecture}
It is known that this conjecture is equivalent to a deep conjecture in arithmetic geometry, the Weak Ordinarity conjecture (see \cite{Mustata} and \cite{MS}, as well as
\cite{BST} for the case of a singular ambient variety). It is instructive to see what the above theorem and conjecture say about the relation between $\lct(f)$ and $\fpt(f_t)$. 
First, the assertion in Theorem~\ref{thm0}i) says that after possibly replacing $A$ by a localization $A_a$, we may assume that $\fpt(f_t)\leq \lct(f)$ for all closed points in $t\in {\rm Spec}(A)$.
On the other hand, the result in Theorem~\ref{thm0}ii) implies that
\begin{equation}\label{eq_limit_intro}
\lim_{{\rm char}\,k(t)\to\infty}\fpt(f_t)=\lct(f).
\end{equation}
Finally, Conjecture~\ref{conj1} predicts that there is a dense subset $T$ of closed points in ${\rm Spec}(A)$ such that
$\fpt(f_t)=\lct(f)$ for all $t\in T$. 

Our main goal in this note is to give an effective estimate of the limit in (\ref{eq_limit_intro}) and a generalization of this to arbitrary jumping numbers, in terms of the roots of the \emph{Bernstein-Sato polynomial}
$b_f(s)$ of $f$. Recall that $b_f(s)$ is the monic polynomial of smallest degree such that
\begin{equation}\label{eq_b_f}
b_f(s)f^s\in\cD_X[s]\bullet f^{s+1},
\end{equation}
where $\cD_X$ is the sheaf of differential operators on $X$. Here $f^s$ has to be interpreted as a formal symbol on which differential operators
act in the expected way. The existence of a polynomial $b_f(s)$ as in (\ref{eq_b_f}) was proved by Bernstein in \cite{Bernstein} when $X={\mathbf A}^n$
and by Kashiwara \cite{Kashiwara} in general (in fact, for arbitrary holomorphic functions on complex manifolds). Furthermore, it was shown in \cite{Kashiwara}
that all roots of $b_f$ are negative rational numbers.

The following is our estimate for the $F$-pure thresholds of the reductions of $f$. 

\begin{thm}\label{thm1}
Let $X$ be a smooth complex algebraic variety and $f\in\cO_X(X)$ nonzero. If $(X_A, f_A)$ gives a model of $(X, f)$ over $A$, then after possibly replacing $A$ by a localization
$A_a$, for every closed point $t\in {\rm Spec}(A)$ with ${\rm char}\,k(t)=p_t$,
the following holds: if $i$ is a positive integer such that $(p_t-i)\cdot \lct(f)\in\Z$, then 
$$\fpt(f_t)>\lct(f)-\frac{m_i}{p_t},$$
where $m_i=\max\big\{\beta\mid b_f(-\beta)=0, \beta-(i\cdot\lct(f)+1-\lceil i\cdot\lct(f)\rceil)\in\Z_{\geq 0}\big\}$. 
\end{thm}

Note that if we write $\lct(f)=\tfrac{a}{b}$, with $a$ and $b$ relatively prime, then it is enough to consider in the theorem only those $i$ with $1\leq i\leq b$ and relatively prime
to $b$. In particular, we get an explicit lower bound for $\fpt(f_t)$, in terms of the roots of $b_f$,  that works for all $t$. 
Using a lower bound for the roots of the Bernstein-Sato
polynomial from \cite{Saito_microlocal}, we will also see that we may replace $m_i$ in the theorem by $\dim(X)$ 
(see Remark~\ref{uniform_bound}).

In fact, we prove a more general statement (see Theorem~\ref{thm2}), in which $\lct(f)$ is replaced by any jumping number of $f$. The main ingredient in the proof
of Theorem~\ref{thm1}
is a result from \cite{MZ} which says that there is such a lower bound, but for which $m_i$ is only determined, in a complicated way, by a log resolution (this, in turn, relied on the techniques in \cite{HY} that give (\ref{eq_limit_intro})). The other ingredient is the elementary observation, which goes back to \cite{MTW}, that for every 
closed point $t\in {\rm Spec}(A)$, the integer 
$\lceil \fpt(f_t)p_t\rceil-1$ gives a root of $b_f$ mod $p_t$. 

We note that the assertion in Theorem~\ref{thm1} when $\lct(f)=1$ has been recently obtained by Dodd in \cite{Dodd}, using completely different methods. In fact, his result 
provided us with the motivation for revisiting this circle of ideas.

It is well-known that $-\lct(f)$ is a root of $b_f$ (in fact, this is the largest root of the Bernstein-Sato polynomial, see \cite[Section~10]{Kollar}). When $b_f$ has no roots
of the form $-\lct(f)-n$, with $n\in\Z_{>0}$, we obtain the following result in the direction of the conjectural existence of prime reductions with $F$-pure threshold equal to 
$\lct(f)$.

\begin{cor}\label{cor_thm1}
Let $X$ be a smooth complex algebraic variety and $f\in\cO_X(X)$ nonzero such that $b_f\big(-\lct(f)-n\big)\neq 0$ for every $n\in\Z_{>0}$. 
If $(X_A,f_A)$ gives a model of $(X,f)$, then after possibly replacing $A$ by a localization
$A_a$, for every closed point $t\in {\rm Spec}(A)$ with ${\rm char}\,k(t)=p_t$, we have
$\fpt(f_p)=\lct(f)$ as long as $(p_t-1)\cdot \lct(f)\in\Z$.
\end{cor}

\medskip

\noindent {\bf Acknowledgment}. I am grateful to Christopher Dodd for sharing with me a preliminary version of \cite{Dodd} and for several discussions
on this topic. I would also like to thank Shunsuke Takagi for some good questions on the first version of this note. Last but not least, I am indebted to the anonymous referee for several useful suggestions, especially for Example~\ref{example_referee}.

\section{A brief review of $F$-jumping numbers}

Let $Y$ be a regular $F$-finite scheme of characteristic $p>0$ and let $g\in\cO_Y(Y)$ be everywhere 
nonzero (in other words, $g$ is nonzero on every connected component of $Y$). 
Note that since $Y$ is regular, the Frobenius morphism $F\colon X\to X$ is also flat by a famous result of Kunz \cite{Kunz}. 
We do not recall the definition of the test ideals
$\tau(g^{\lambda})$ since we will not need it. We refer the reader to \cite{HY} for the original definition and to \cite{BMS1} for a simpler
one in our setting (when the ambient scheme is assumed to be regular and $F$-finite). 

We only review here the description of the $F$-jumping numbers of $g$ as $F$-thresholds, following \cite{MTW} and \cite{BMS1}. 
We note that if $Y=U_1\cup\ldots\cup U_N$ is an open cover, then the set of $F$-jumping numbers of $g$ is the union of the sets
of the jumping numbers of $g\vert_{U_i}$, for $1\leq i\leq N$. By taking a suitable affine cover of $Y$, we may thus reduce the study
of $F$-jumping numbers to the case of an affine scheme. 
In order to conform to the presentation in \cite{MTW} and \cite{BMS1}, we assume that $Y$ is affine and let $R=\cO_Y(Y)$.

Given a proper ideal $J\subseteq R$ such that $g\in {\rm rad}(J)$, for every $e\geq 1$ we put
$$\nu^J_g(p^e)=\max\{r\in\Z_{\geq 0}\mid g^r\not\in J^{[p^e]}\}.$$
Recall that $J^{[p^e]}=(h^{p^e}\mid h\in J)$. 
The flatness of the Frobenius morphism gives 
$$\big(J^{[p^{e+1}]} : g^{pr}\big)=\big(J^{[p^e]} : g^r\big)^{[p]},$$
hence $g^r\not\in J^{[p^e]}$ implies $g^{pr}\not\in J^{[p^{e+1}]}$.
We thus have 
$$\frac{\nu^J_g(p^e)}{p^e}\leq\frac{\nu^J_g(p^{e+1})}{p^{e+1}}\quad\text{for all}\quad e\geq 1$$
and the \emph{$F$-threshold of $g$ with respect to $J$} is
$$c^J(g):=\sup_{e\geq 1}\frac{\nu^J_g(p^e)}{p^e}=\lim_{e\to\infty}\frac{\nu^J_g(p^e)}{p^e}.$$

One can show that $c^J(g)<\infty$ and we have
\begin{equation}\label{eq1_F}
\frac{\nu^J_g(p^e)}{p^e}<c^J(g)\quad\text{for all}\quad e\geq 1
\end{equation}
(see \cite[Remark~1.2]{MTW} and \cite[Proposition~1.7(5)]{MTW}). 
What's special when working with a principal ideal such as $(g)$, is that if $g^{r+1}\in J^{[p^e]}$, then
$g^{(r+1)p}\in J^{[p^{e+1}]}$. We thus have
$$\frac{\nu^J_g(p^e)+1}{p^e}\geq \frac{\nu^J_g(p^{e+1})}{p^{e+1}}\quad\text{for all}\quad e\geq 1,$$
which immediately implies
\begin{equation}\label{eq2_F}
c^J(g)\leq\frac{\nu^J_g(p^e)+1}{p^e}\quad\text{for all}\quad e\geq 1.
\end{equation}
By combining the inequalities (\ref{eq1_F}) and (\ref{eq2_F}), we obtain
\begin{equation}\label{eq_nu}
\nu^J_g(p^e)+1=\lceil c^J(g)\cdot p^e\rceil\quad\text{for all}\quad e\geq 1
\end{equation}
(see \cite[Proposition~1.9]{MTW}). 

The $F$-thresholds are relevant for us since the set of $F$-jumping numbers of $g$ coincides with the set of all $F$-thresholds of $g$,
when the ideal $J$ varies (see \cite[Corollary~2.30]{BMS1}). We note that unlike in \emph{loc}. \emph{cit}. we do not consider $0$ as an $F$-jumping
number, which corresponds to the fact that we require $J$ to be a proper ideal of $\cO_X$. 

As we have mentioned in the Introduction, a basic result about the $F$-jumping numbers of $g$ is that they form a discrete set of rational numbers 
(see \cite[Theorem~3.1]{BMS1} for the case when $R$ is essentially of finite type over a field and
\cite[Theorem~1.1]{BMS2} for the general case). We note that the corresponding result about jumping numbers in characteristic 0 is an immediate consequence
of their description in terms of a log resolution, see \cite[Lemma~9.3.21]{Lazarsfeld}. 

As it is the case for jumping numbers in characteristic $0$, some $\lambda>1$ is an $F$-jumping number of $g$ if and only if $\lambda-1$ has the same property.
This follows, for example, from the fact that $\tau(g^{\alpha})=g\cdot\tau(g^{\alpha-1})$ for every $\alpha>1$ (see for example 
\cite[Proposition~2.25]{BMS1}).
A more interesting fact, that is peculiar to positive characteristic, is that if $\lambda$ is an $F$-jumping number of $g$, then so is $p\lambda$. This follows from the fact that
$p\cdot c^J(g)=c^{J^{[p]}}(g)$, see \cite[Proposition~3.4(1)]{BMS1}.
By combining these two facts one can show that the $F$-pure threshold $\fpt(g)$ can't lie in certain intervals: more precisely, 
for every $e\geq 1$ and every integer $a$, with $1\leq a\leq p^e-1$, we have
\begin{equation}\label{eq_fpt_interval}
\fpt(g)\not\in \left(\tfrac{a}{p^e},\tfrac{a}{p^e-1}\right).
\end{equation}
This is due to the fact that if $\fpt(g)\in \big(\tfrac{a}{p^e},\tfrac{a}{p^e-1}\big)$, then $p^e\cdot \fpt(g)-a$ is an $F$-jumping number of $g$ that is $<\fpt(g)$, a contradiction; 
see \cite[Proposition~4.3]{BMS2} for details.

We end this discussion of $F$-jumping numbers with the following result that we will need:

\begin{prop}\label{thm_ext_test}
Suppose that $Y={\rm Spec}(R)$ is a smooth affine scheme of finite type over the finite field $k$ and $k\subseteq k'$ is a finite field extension. If $g\in R$ is nowhere zero
and $g'=g\otimes 1\in R'=R\otimes_kk'$, then
$$\tau({g'}^{\lambda})=\tau(g^{\lambda})\cdot R'\quad\text{for all}\quad \lambda\in\R_{\geq 0}.$$
In particular, $g$ and $g'$ have the same $F$-jumping numbers.
\end{prop}

The first assertion is a very special case of \cite[Theorem~3.3]{HT}, which applies since the homomorphism $R\to R'$ is finite and \'{e}tale. The second assertion is an immediate consequence.

\section{The proof of the main result}

Our goal is to prove the following more general version of Theorem~\ref{thm1}.

\begin{thm}\label{thm2}
Let $X$ be a smooth, irreducible complex algebraic variety, $f\in\cO_X(X)$ nonzero, $\lambda>0$ a jumping number of $f$, and $\lambda'<\lambda$ such that
$\cJ(f^{\alpha})$ takes the same value for all $\alpha\in [\lambda',\lambda)$. If $(X_A,f_A)$ gives a model of $(X,f)$ over $A$,
then after possibly replacing $A$ by a localization
$A_a$, for every closed point $t\in {\rm Spec}(A)$ with ${\rm char}\,k(t)=p_t$, there is an $F$-jumping number $\mu$ for $f_t$ in the interval $(\lambda',\lambda]$,
and for every such $\mu$, if $(p_t-i)\lambda\in\Z$, then 
$$\mu>\lambda-\frac{m_i}{p_t},$$
where $m_i=\max\big\{\beta\mid b_f(-\beta)=0, \beta-(\lambda i+1-\lceil \lambda i\rceil)\in\Z_{\geq 0}\}$.
\end{thm}

\begin{rmk}\label{uniform_bound}
In the setting of Theorem~\ref{thm2}, if $\lambda=\tfrac{a}{b}$, with $a$, $b$ relatively prime, and 
$$I=\{i\mid 1\leq i\leq b,\, {\rm gcd}(i,b)=1\},$$
then for every $t$ there is $i\in I$ 
such that $(p_t-i)\lambda\in\Z$. The assertion in the theorem thus implies that if $m=\max_{i\in I}m_i$, then $\mu>\lambda-\frac{m}{p_t}$.

In fact, we can get a uniform bound just in terms of $n=\dim(X)$. More precisely, it follows from 
\cite{Saito_microlocal} that all roots of $b_f(s)$ are $>-n$; therefore, in the situation in the theorem, we have
$\mu>\lambda-\tfrac{n}{p_t}$ for all $t$. 
\end{rmk}

Note that by taking $\lambda=\lct(f)$ in Theorem~\ref{thm2}, we deduce the assertion in Theorem~\ref{thm1}. Hence from now on we focus on Theorem~\ref{thm2}.

\begin{rmk}
If $X$, $f$, and $A$ are as in the theorem, and $\lambda\in\R_{>0}$ is not a jumping number of $f$, then there is $\lambda'<\lambda$ such that 
$\cJ(f^{\lambda})=\cJ(f^{\lambda'})$. Applying Theorem~\ref{thm0} for both $\lambda$ and $\lambda'$, we see that if $t\in {\rm Spec}(A)$ is a closed point
with ${\rm char}\,k(t)\gg 0$, then there is no $F$-jumping number of $f_t$ in the interval $(\lambda',\lambda]$. 
\end{rmk}

Before giving the proof of Theorem~\ref{thm2}, we make a few preliminary remarks concerning models for $(X,f)$ over $A$. This is standard material, for more details we refer to
\cite[Section~2.2]{MS}. 

The assumption in Theorem~\ref{thm2} is that $A\subseteq\C$ is a finite type algebra over $\Z$ and $X_A$ is a scheme of finite type over $A$ 
and $f_A\in \cO_{X_A}(X_A)$ are such that we have an isomorphism $X_A\times_A\C\simeq X$ such that the pull-back of $f_A$ is mapped to $f$. By generic
smoothness, after possibly replacing $A$ by a localization $A_a$, we may and always assume that $X_A$ is smooth over ${\rm Spec}(A)$ and $f_A$ defines a relative Cartier divisor
in $X_A$ over ${\rm Spec}(A)$. 

If $A\subseteq A'\subseteq \C$, where $A'$ is another finite type algebra over $\Z$, then we have $X_{A'}=X\times_AA'$ and the image $f_{A'}$ of $f_A$ in 
$\cO_{X_{A'}}(X_{A'})$, which give a model for $(X,f)$ over $A'$.
By general properties of finite type morphisms, 
 there is a nonzero $a\in A$ such that ${\rm Spec}(A_a)$
is contained in the image of ${\rm Spec}(A')$. Moreover, if $t\in {\rm Spec}(A_a)$ is a closed point, then there is a closed point $t'\in {\rm Spec}(A')$
whose image is $t$. In particular, the field extension $k(t)\hookrightarrow k(t')$ is an extension of finite fields, hence we may
apply
Proposition~\ref{thm_ext_test} to the elements of an affine open cover of $X_t$ to conclude
that $f_t$ and $f_{t'}$ have the same $F$-jumping numbers. Therefore it is enough to prove the theorem for $X_{A'}$ and $f_{A'}$, hence we are free to replace $A$ by a larger
algebra with the same properties. 

We can now prove our main result.

\begin{proof}[Proof of Theorem~\ref{thm2}]
If $X=U_1\cup\ldots\cup U_N$ is an affine open cover, then $b_f={\rm lcm}(b_f\vert_{U_i})$. Since the set of jumping numbers of $f$ is the union of the sets of jumping
numbers of $f\vert_{U_i}$, for $1\leq i\leq N$, it is straightforward to see that it is enough to prove the assertion in the theorem for each $U_i$. Hence from now on
we may and will assume that $X={\rm Spec}(R)$ is affine. 

Suppose that $X_A={\rm Spec}(R_A)$. For every closed point  $t\in {\rm Spec}(A)$, let $R_t=R_A\otimes_Ak(t)$ be the $k(t)$-algebra corresponding to $X_t$.
By definition of the Bernstein-Sato polynomial, there is $P(s)\in D_R[s]$ such that
$$b_f(s)f^s=P(s)\bullet f^{s+1},$$
where $D_R$ is the ring of differential operators of $R$. 
After possibly replacing $A$ by a larger finitely generated $\Z$-subalgebra of $\C$, we may assume that 
the denominators of the roots of $b_f$ are invertible in $A$ and that we have $P_A(s)\in D^{(0)}_{R_A/A}[s]$
such that
$$b_f(s)f_A^s=P_A(s)\bullet f_A^{s+1},$$
where $D_{R_A/A}^{(0)}\subseteq {\rm End}_A(R_A)$ is the subring generated by $R_A$ and ${\rm Der}_A(R_A)$. 
In this case, for every closed point $t\in {\rm Spec}(A)$ with ${\rm char}\,k(t)=p_t$, we get a corresponding relation
$$\overline{b_f}(s)f_t^s=P_t(s)\bullet f_t^{s+1},$$
where $\overline{b_f}\in {\mathbf F}_{p_t}[s]$ is the image of $b_f$ and $P_t(s)$ has coefficients in the subring $D_{R_t/k(t)}^{(0)}\subseteq {\rm End}_{k(t)}(R_t)$
generated by $R_t$ and ${\rm Der}_{k(t)}(R_t)$. In particular, for every $m\in \Z_{\geq 0}$, we have
\begin{equation}\label{eq_b_funct}
\overline{b_f}(m)f_t^m\in D_{R_t/k(t)}^{(0)}\bullet f_t^{m+1}\subseteq R_t.
\end{equation}
A key observation, going back to \cite[Proposition~3.11]{MTW} is that for every proper ideal $J$ in $R_t$ and every $e\geq 1$, we have
\begin{equation}\label{eq_cong}
\overline{b_f}\big(\nu_{f_t}^J(p_t^e)\big)=0\,\,\text{in}\,\,{\mathbf F}_{p_t}.
\end{equation}
Indeed, if we take $m=\nu^J_{f_t}(p_t^e)$, then by definition we have $f^m\not\in J^{[p_t^e]}$, while $f^{m+1}\in J^{[p_t^e]}$. Since the ideal $J^{[p_t^e]}$ 
is preserved by the action of $D_{R_t/k(t)}^{(0)}$, the assertion in (\ref{eq_cong}) is a consequence of the formula in (\ref{eq_b_funct}). 

The other main ingredient in our proof is \cite[Theorem~B(i)]{MZ}, which says that there is $C>0$ such that after possibly replacing $A$ by some localization $A_a$,
we may assume that for every closed point $t\in {\rm Spec}(A)$ with ${\rm char}\,k(t)=p_t$, there is always an $F$-jumping number $\mu\in (\lambda',\lambda]$ for $f_t$
and every such $F$-jumping number satisfies 
\begin{equation}\label{eq_MZ}
\lambda-\tfrac{C}{p_t}\leq \mu\leq\lambda.
\end{equation}
Without any loss of generality, we may and will assume that $C\in\Z$, in which case it follows from (\ref{eq_MZ}) that 
\begin{equation}\label{eq_MZ2}
0\leq \lceil \lambda p_t\rceil-\lceil\mu p_t\rceil\leq C.
\end{equation}

Let us pick one such $t\in {\rm Spec}(A)$ and $\mu$ as above. 
Using the description of $F$-jumping numbers as $F$-thresholds discussed in the previous section, we see that there is
a proper ideal $J$ in $R_t$ such that $\mu=c^J(f_t)$.  In this case, it follows from (\ref{eq_nu}) that $\lceil \mu p_t\rceil =\nu^J_{f_t}(p_t)+1$.
Let us choose a positive integer $d$ such that the coefficients of $b_f$ lie in $\tfrac{1}{d}\Z$ and $d\lambda^i\in\Z$ for $i\leq {\rm deg}(b_f)$. 
We may and will assume that $p_t$ does not divide $d$. We deduce from 
(\ref{eq_cong}) that 
\begin{equation}\label{eq_cong2}
b_f\big(\lceil \mu p_t\rceil-1\big)\equiv 0\,\,({\rm mod}\,\,p_t)\,\,\text{in}\,\,\tfrac{1}{d}\Z.
\end{equation}

Let $k=\lceil \lambda p_t\rceil-\lceil\mu p_t\rceil$, so $0\leq k\leq C$ by (\ref{eq_MZ2}). 
Since we can write
$\lambda p_t=\lambda(p_t-i)+\lambda i$ and $\lambda(p_t-i)\in\Z$, it follows that
$$\lceil \lambda p_t\rceil =\lambda(p_t-i)+\lceil \lambda i\rceil.$$
The condition (\ref{eq_cong2}) thus implies
$$
b_f\big(\lambda (p_t-i)+\lceil \lambda i\rceil-k-1\big)\equiv 0\,\,({\rm mod}\,\,p_t)\,\,\text{in}\,\,\tfrac{1}{d}\Z,
$$
and thus
\begin{equation}\label{eq_cong3}
b_f\big(\lceil \lambda i\rceil-\lambda i-k-1)\equiv 0\,\,({\rm mod}\,\,p_t)\,\,\text{in}\,\,\tfrac{1}{d}\Z.
\end{equation}
Since $k$ is an integer bounded independently of $t$ and $\mu$, it can take only finitely many values. After possibly inverting finitely many prime integers,
we may assume that every $k$ that appears when we vary $t\in {\rm Spec}(A)$, appears for infinitely many primes $p_t$. In this case, condition
(\ref{eq_cong3}) implies that for every such $k$, we have
\begin{equation}\label{eq_cong4}
b_f\big(\lceil \lambda i\rceil-\lambda i-k-1)=0. 
\end{equation}
By definition, we thus have 
\begin{equation}\label{final_eq}
k+1+\lambda i-\lceil\lambda i\rceil\leq m_i.
\end{equation}
Note now that we have
$$\mu>\frac{\lceil\mu p_t\rceil-1}{p_t}=\frac{\lceil\lambda p_t\rceil-k-1}{p_t}=
\frac{\lambda(p_t-i)+\lceil\lambda i\rceil-k-1}{p_t}\geq \lambda-\frac{m_i}{p_t},$$
where the last inequality follows from (\ref{final_eq}).
This completes the proof of the theorem.
\end{proof}

We next deduce our result on the equality between $\lct(f)$ and $\fpt(f_t)$.

\begin{proof}[Proof of Corollary~\ref{cor_thm1}]
It follows from Theorem~\ref{thm1} that under our assumption, we may assume that for every closed point $t\in {\rm Spec}(A)$ such that $(p_t-1)\cdot\lct(f)\in\Z$, we have
$$\fpt(f_t)>\lct(f)-\frac{\lct(f)}{p_t},$$
where $p_t={\rm char}\,k(t)$. Moreover, we also have
$$\fpt(f_t)\leq\lct(f)$$
by Theorem~\ref{thm0}i). In order to prove that $\fpt(f_t)=\lct(f)$, it is enough to show that we can't have $\fpt(f_t)\in\left(\lct(f)-\tfrac{\lct(f)}{p_t},\lct(f)\right)$. 
This follows from (\ref{eq_fpt_interval}). Indeed, let us write $\lct(f)=\tfrac{a}{b}$, with ${\rm gcd}(a,b)=1$.
Since $(p_t-1)\cdot\lct(f)\in\Z$, we can write $p_t-1=bc$ for some positive integer $c$. We then have
$$\lct(f)=\frac{ac}{p_t-1}\quad\text{and}\quad \lct(f)-\frac{\lct(f)}{p_t}=\frac{ac}{p_t},$$
hence (\ref{eq_fpt_interval}) gives the assertion we need.
\end{proof}

\begin{eg}
Let $f=x^2+y^3\in\C[x,y]$. It is well-known that $\lct(f)=\tfrac{5}{6}$ 
(see \cite[Example~9.2.15]{Lazarsfeld}) 
and in fact
$$b_f(s)=\left(s+\tfrac{5}{6}\right)(s+1)\left(s+\tfrac{7}{6}\right)$$
(see for example \cite[Example~6.19]{Kashiwara2}). In this case we can take $A=\Z$ and the model given by $(\A_{\Z}^2, f)$, where we view $f$ as an element of 
$\Z[x,y]$.
If we denote by $f_p$ the image of $f$ in ${\mathbf F}_p[x,y]$, with $p\gg 0$, then it follows from Corollary~\ref{cor_thm1} that if $p\equiv 1$ (mod $6$),
then $\fpt(f_p)=\lct(f)=\tfrac{5}{6}$, and it follows from Theorem~\ref{thm1} that if $p\equiv 5$ (mod $6$), then
$$\fpt(f_p)\geq \tfrac{5}{6}-\tfrac{7}{6p}.$$
In fact, it is known that in this case $\fpt(f_p)=\tfrac{5}{6}-\tfrac{1}{6p}$, see
\cite[Example~4.3]{MTW}. 
\end{eg}

The following example was pointed out by the anonymous referee.

\begin{eg}\label{example_referee}
Consider the polynomials $f,g\in {\mathbf C}[x,y]$, where
$$f=x^5+y^4\quad\text{and}\quad g=x^5+x^3y^2+y^4.$$
The roots of $b_f(s)$ are the negatives of
$$\tfrac{9}{20}, \tfrac{13}{20}, \tfrac{7}{10}, \tfrac{17}{20}, \tfrac{9}{10},\tfrac{19}{20}, 1, \tfrac{21}{20},\tfrac{11}{10}, \tfrac{23}{20},\tfrac{13}{10},\tfrac{27}{20},
{\bf \tfrac{31}{20}}$$
and the roots of $b_g(s)$ are the negatives of 
$$\tfrac{9}{20}, {\bf \tfrac{11}{20}}, \tfrac{13}{20}, \tfrac{7}{10}, \tfrac{17}{20}, \tfrac{9}{10},\tfrac{19}{20}, 1, \tfrac{21}{20},\tfrac{11}{10}, \tfrac{23}{20},
\tfrac{13}{10},\tfrac{27}{20}$$
(see \cite[Sections~11 and 18]{Yano}).
This is a well-known example in which $f$ and $g$ are part of a family of isolated singularities, with constant Milnor number, but such that the
Bernstein-Sato polynomials are different.
In both cases we may take $A={\mathbf Z}$ and the models given by $({\mathbf A}^2_{\mathbf Z},f)$ and $({\mathbf A}^2_{\mathbf Z},g)$,
respectively, where we view $f$ and $g$ as elements of ${\mathbf Z}[x,y]$.
 Note that we have ${\rm lct}(f)=\tfrac{9}{20}={\rm lct}(g)$ and
 $$\fpt(f_p)=\tfrac{9}{20}\,\,\text{if}\,\,p\equiv 1\,(\text{mod}\,20)\quad\text{and}\quad \fpt(f_p)=\tfrac{9p-11}{20p}\,\,\text{if}\,\,p\equiv 19\,(\text{mod}\,20)$$
 (this follows, for example, from \cite[Theorem~3.1]{Hernandez}), while
 $$\fpt(g_p)=\tfrac{9}{20}\,\,\text{if}\,\,p\equiv 1\,(\text{mod}\,20)\quad\text{and}\quad \fpt(g_p)=\tfrac{9p-11}{20(p-1)}\,\,\text{if}\,\,p\equiv 19\,(\text{mod}\,20)$$
 (see \cite[Example~4.5]{MTW}). 
 Note that this matches the predictions provided by Corollary~\ref{cor_thm1}, namely that
 $$\fpt(f_p)=\lct(f)\quad\text{and}\quad \fpt(g_p)=\lct(g)\quad\text{if}\quad p\equiv 1\,(\text{mod}\,20)$$
 and those provided by Theorem~\ref{thm1}, namely that 
 $$\fpt(f_p)>\lct(f)-\tfrac{31}{20p}\quad\text{and}\quad \fpt(g_p)>\lct(g)-\tfrac{11}{20p}\quad\text{if}\quad p\equiv 19\,(\text{mod}\,20).$$
Note that in this case the bound satisfied by $g$ when $p\equiv 19$ (mod $20$) is \emph{not} satisfied by $f$ (in fact, for $f$ this becomes equality),
but this is allowed due to the presence of the root $-\tfrac{31}{20}$ of $b_f(s)$.
\end{eg}

\begin{rmk}
Suppose now that $X$ is a smooth, irreducible, $n$-dimensional complex algebraic variety and $\fra$ is an arbitrary nonzero coherent ideal sheaf on $X$. 
Recall that we can associate multiplier ideals and test ideals to non-principal ideal as well (see \cite[Chapter~9]{Lazarsfeld} for the case of multiplier ideals and 
\cite{BMS1} for the case of test ideals)
Let $\lambda>0$ be a jumping number of $\fra$, and let $\lambda'<\lambda$ be such that
$\cJ(\fra^{\alpha})$ takes the same value for all $\alpha\in [\lambda',\lambda)$. 
We fix an integer $d>\lambda$. 
We claim that 
if $(X_A,\fra_A)$ gives a model of $(X,\fra)$,
then after possibly replacing $A$ by a localization
$A_a$, for every closed point $t\in {\rm Spec}(A)$ with ${\rm char}\,k(t)=p_t$, there is an $F$-jumping number $\mu$ for $\fra_t$ in the interval $(\lambda',\lambda]$,
and for every such $\mu$, we have
\begin{equation}\label{last_eq}
\mu>\lambda-\frac{dn}{p_t}.
\end{equation}
Indeed, in order to see this, we may again assume that $X={\rm Spec}(R)$ is affine and let $f_1,\ldots,f_r\in R$ be generators of $\fra$. 
For $1\leq i\leq d$, let $h_i$ be a general linear combination of $f_1,\ldots,f_r$, with $\C$-coefficients, and let $h=\prod_{i=1}^dh_i$.
In this case it follows from \cite[Proposition~9.2.28]{Lazarsfeld} that 
\begin{equation}\label{eq_Laz}
\cJ(\fra^{\alpha})=\cJ(h^{\alpha/d})\quad\text{for all}\quad \alpha\leq \lambda.
\end{equation}
After possibly enlarging $A$, we may assume that we have a model $h_A\in R_A$ for $h$, such that $h_A\in\fra_A^d$.

We may assume the existence of an $F$-jumping number for $\fra_t$ in the interval $(\lambda',\lambda]$ 
by Theorem~\ref{thm0} (more precisely, by its version for arbitrary ideals), hence we only need to prove that 
\begin{equation}\label{last_eq2}
\tau(\fra_t^{\lambda'})=\tau(\fra_t^{\lambda-\tfrac{dn}{p_t}}).
\end{equation}
Note that we may assume that we have the equalities
\begin{equation}\label{eq100}
\tau(\fra_t^{\lambda'})=\cJ(\fra^{\lambda'})_t=\cJ(h^{\tfrac{\lambda'}{d}})_t=\tau(h_t^{\tfrac{\lambda'}{d}})=\tau(h_t^{\tfrac{\lambda}{d}-\tfrac{n}{p_t}}),
\end{equation}
where the first and third equalities follow from (\ref{thm0}), the second one follows from (\ref{eq_Laz}), and the fourth one follows from 
Remark~\ref{uniform_bound}. 
On the other hand, we have the inclusions
\begin{equation}\label{eq101}
\tau(h_t^{\tfrac{\lambda}{d}-\tfrac{n}{p_t}})\subseteq \tau(\fra_t^{\lambda-\tfrac{dn}{p_t}})\subseteq \tau(\fra_t^{\lambda'}),
\end{equation}
where the first inclusion follows from $h_t\in\fra_t^d$ and the second one follows from $\lambda'\leq\lambda-\tfrac{dn}{p_t}$.
By combining (\ref{eq100}) and (\ref{eq101}), we obtain, in particular, the equality (\ref{last_eq2}). This completes the proof of our assertion. 

We finally note that for an arbitrary ideal $\fra$ there is a notion of Bernstein-Sato polynomial, whose
roots are negative rational numbers, see \cite{BMS}. It is an interesting question whether 
the bound in (\ref{last_eq}) can be refined by taking into account these roots as in Theorem~\ref{thm2}. In fact, everything in the proof of the theorem
carries through in this setting, with the exception of one point: it is not clear that having $\lambda p_t-C\leq\mu p_t$ implies that $\lambda p_t-\nu^J_{\fra}(p_t)$
stays bounded.
\end{rmk}


\section*{References}
\begin{biblist}

\bib{Bernstein}{article}{
   author={Bern\v{s}te\u{\i}n, I. N.},
   title={Analytic continuation of generalized functions with respect to a
   parameter},
   journal={Funkcional. Anal. i Prilo\v{z}en.},
   volume={6},
   date={1972},
   number={4},
   pages={26--40},
}

\bib{BST}{article}{
   author={Bhatt, B.},
   author={Schwede, K.},
   author={Takagi, S.},
   title={The weak ordinarity conjecture and $F$-singularities},
   conference={
      title={Higher dimensional algebraic geometry---in honour of Professor
      Yujiro Kawamata's sixtieth birthday},
   },
   book={
      series={Adv. Stud. Pure Math.},
      volume={74},
      publisher={Math. Soc. Japan, Tokyo},
   },
   date={2017},
   pages={11--39},
}


\bib{BMS1}{article}{
   author={Blickle, M.},
   author={Musta\c{t}\v{a}, M.},
   author={Smith, K.~E.},
   title={Discreteness and rationality of $F$-thresholds},
   note={Special volume in honor of Melvin Hochster},
   journal={Michigan Math. J.},
   volume={57},
   date={2008},
   pages={43--61},
}

\bib{BMS2}{article}{
   author={Blickle, M.},
   author={Musta\c{t}\u{a}, M.},
   author={Smith, K.~E.},
   title={$F$-thresholds of hypersurfaces},
   journal={Trans. Amer. Math. Soc.},
   volume={361},
   date={2009},
   number={12},
   pages={6549--6565},
}


\bib{BMS}{article}{
   author={Budur, N.},
   author={Musta\c{t}\v{a}, M.},
   author={Saito, M.},
   title={Bernstein-Sato polynomials of arbitrary varieties},
   journal={Compos. Math.},
   volume={142},
   date={2006},
   number={3},
   pages={779--797},
}


\bib{Dodd}{article}{
author= {Dodd, C.},
title={Differential operators, gauges, and mixed Hodge modules}, 
journal={preprint arXiv:2210.12611},
date={2022},
}

\bib{HT}{article}{
   author={Hara, N.},
   author={Takagi, S.},
   title={On a generalization of test ideals},
   journal={Nagoya Math. J.},
   volume={175},
   date={2004},
   pages={59--74},
}

\bib{HY}{article}{
   author={Hara, N.},
   author={Yoshida, K.-I.},
   title={A generalization of tight closure and multiplier ideals},
   journal={Trans. Amer. Math. Soc.},
   volume={355},
   date={2003},
   number={8},
   pages={3143--3174},
}


\bib{Hernandez}{article}{
   author={Hern\'{a}ndez, D.~J.},
   title={$F$-invariants of diagonal hypersurfaces},
   journal={Proc. Amer. Math. Soc.},
   volume={143},
   date={2015},
   number={1},
   pages={87--104},
}

\bib{Kashiwara}{article}{
   author={Kashiwara, M.},
   title={$B$-functions and holonomic systems. Rationality of roots of
   $B$-functions},
   journal={Invent. Math.},
   volume={38},
   date={1976/77},
   number={1},
   pages={33--53},
}

\bib{Kashiwara2}{book}{
   author={Kashiwara, M.},
   title={$D$-modules and microlocal calculus},
   series={Translations of Mathematical Monographs},
   volume={217},
   note={Translated from the 2000 Japanese original by Mutsumi Saito;
   Iwanami Series in Modern Mathematics},
   publisher={American Mathematical Society, Providence, RI},
   date={2003},
   pages={xvi+254},
}

\bib{Kollar}{article}{
   author={Koll\'ar, J.},
   title={Singularities of pairs},
   conference={
      title={Algebraic geometry---Santa Cruz 1995},
   },
   book={
      series={Proc. Sympos. Pure Math.},
      volume={62},
      publisher={Amer. Math. Soc., Providence, RI},
   },
   date={1997},
   pages={221--287},
}

\bib{Kunz}{article}{
   author={Kunz, E.},
   title={Characterizations of regular local rings of characteristic $p$},
   journal={Amer. J. Math.},
   volume={91},
   date={1969},
   pages={772--784},
}

\bib{Lazarsfeld}{book}{
       author={Lazarsfeld, R.},
       title={Positivity in algebraic geometry II},  
       series={Ergebnisse der Mathematik und ihrer Grenzgebiete},  
       volume={49},
       publisher={Springer-Verlag, Berlin},
       date={2004},
}      

\bib{Mustata}{article}{
   author={Musta\c{t}\u{a}, M.},
   title={Ordinary varieties and the comparison between multiplier ideals
   and test ideals II},
   journal={Proc. Amer. Math. Soc.},
   volume={140},
   date={2012},
   number={3},
   pages={805--810},
}

\bib{MS}{article}{
   author={Musta\c{t}\u{a}, M.},
   author={Srinivas, V.},
   title={Ordinary varieties and the comparison between multiplier ideals
   and test ideals},
   journal={Nagoya Math. J.},
   volume={204},
   date={2011},
   pages={125--157},
}

\bib{MTW}{article}{
   author={Musta\c{t}\v{a}, M.},
   author={Takagi, S.},
   author={Watanabe, K.-i.},
   title={F-thresholds and Bernstein-Sato polynomials},
   conference={
      title={European Congress of Mathematics},
   },
   book={
      publisher={Eur. Math. Soc., Z\"{u}rich},
   },
   date={2005},
   pages={341--364},
}


\bib{MZ}{article}{
   author={Musta\c{t}\u{a}, M.},
   author={Zhang, W.},
   title={Estimates for $F$-jumping numbers and bounds for
   Hartshorne-Speiser-Lyubeznik numbers},
   journal={Nagoya Math. J.},
   volume={210},
   date={2013},
   pages={133--160},
}


\bib{Saito_microlocal}{article}{
   author={Saito, M.},
   title={On microlocal $b$-function},
   journal={Bull. Soc. Math. France},
   volume={122},
   date={1994},
   number={2},
   pages={163--184},
}


\bib{TW}{article}{
   author={Takagi, S.},
   author={Watanabe, K.-i.},
   title={On F-pure thresholds},
   journal={J. Algebra},
   volume={282},
   date={2004},
   number={1},
}

\bib{Yano}{article}{
   author={Yano, T.},
   title={On the theory of $b$-functions},
   journal={Publ. Res. Inst. Math. Sci.},
   volume={14},
   date={1978},
   number={1},
   pages={111--202},
}



\end{biblist}
\end{document}